\documentclass{amsart}
\usepackage{amsmath,amsthm,amsfonts,amssymb,eucal, hyperref}

\usepackage{mathrsfs}

\renewcommand {\a}{ \alpha }

\newcommand{\e}{\epsilon}

\newcommand{\g}{\gamma}
\newcommand{\G}{\Gamma}

\renewcommand{\d}{\delta}
\newcommand{\s}{\sigma}
\renewcommand{\l}{\lambda}
\renewcommand{\L}{\Lambda}

\renewcommand{\t}{\theta}

\newcommand{\p}{\partial}

\newcommand{\Om}{\Omega}

\newcommand{\R}{ \mathbb R}

\newcommand {\gb}{\mathfrak b}

\newcommand{\gn}{\mathfrak n}
\newcommand{\gr}{\mathfrak r}

\newcommand {\gm}{\mathfrak m}

\newcommand {\GB}{\mathfrak B}

\newcommand {\GM}{\mathfrak M}

\newcommand{\GR}{\mathfrak R}
\newcommand {\GS}{\mathfrak S}

\newcommand {\GN}{\mathfrak N}

\newcommand {\ba}{\mathbf a}

\newcommand {\bx}{\mathbf x}
\newcommand {\bd}{\mathbf d}

\newcommand {\bz}{\mathbf z}
\newcommand {\by}{\mathbf y}

\newcommand{\CK}{\mathcal K}

\newcommand{\CL}{\mathcal L}

\newcommand{\CP}{\mathcal P}

\newcommand{\COB}{\mathscr B}
\newcommand{\COL}{\mathscr L}

\newcommand{\COA}{\mathscr A}

\newcommand{\plainC}[1]{\textup{{\textsf{C}}}^{#1}}

\newcommand{\plainL}[1]{\textup{{\textsf{L}}}^{#1}}

\DeclareMathOperator{\dom}{{Dom}}

\DeclareMathOperator {\re} {{Re}}

 \DeclareMathOperator{\rot}{{curl}}

\DeclareMathOperator{\codim}{{codim}}

\newtheorem{thm}{Theorem}[section]
\newtheorem{cor}[thm]{Corollary}
\newtheorem{lem}[thm]{Lemma}

\theoremstyle{definition}
\numberwithin{equation}{section}

\makeatletter
\def\square{\RIfM@\bgroup\else$\bgroup\aftergroup$\fi
  \vcenter{\hrule\hbox{\vrule\@height.6em\kern.6em\vrule}\hrule}\egroup}
\makeatother

\begin{document}

\title
[Landau operator]
{Discrete spectrum distribution of the Landau Operator
Perturbed by an Expanding Electric Potential }
\author[G. Rozenblum, A. V. Sobolev]{Grigori Rozenblum, Alexander V. Sobolev}
\address{Dept. Mathematics Chalmers University of Technology
and   Dept. Mathematics Gothenburg University, Gothenburg, 41296, Sweden}
\email{grigori@math.chalmers.se}
\address{Department of Mathematics\\
University College London\\
Gower Street\\
London\\
WC1E 6BT, UK}
\email{asobolev@math.ucl.ac.uk}

\dedicatory{Dedicated to our teacher Professor Mikhail
Shl\"emovich Birman on the occasion of his 80-th
 birthday}
\subjclass[2000]{}

\date{\today}

\begin{abstract}

Under a perturbation by a decreasing potential, the Landau Hamiltonian
acquires some discrete eigenvalues between the Landau levels. We study
the perturbation by an "expanding" electric
potential $V(t^{-1}x)$, $t>0$, and derive
a quasi-classical formula for the counting function of the discrete spectrum
as $t\to \infty$.
\end{abstract}

\maketitle
\vskip 0.5cm

\section{Introduction and main result}

%\subsection{Statement of the problem}
The two-dimensional Landau Hamiltonian $H_0  = (-i \nabla - \ba)^2$
describing a charged
 quantum particle moving in the plane in a constant magnetic field $B = \rot \ba$
  is one of the earliest explicitly solvable
models of Quantum Mechanics. Its spectrum consists of
infinitely degenerate eigenvalues (Landau levels)
$\L_q = B(2q+1),\; q=0,1,\dots$,  (see, e.g., \cite{LL}); we put
$\L_{-1} = -\infty$ for reference convenience.
Under a perturbation by an electric or magnetic field decaying at infinity,
the Landau levels split, forming clusters (generically,
 infinite) of eigenvalues with Landau levels being their limit points.
Various asymptotic properties of these clusters have been extensively studied
in the literature.
For instance, \cite{FP}, \cite{melroz03} \cite{raikov90},
\cite{RaiWar} studied the rate of convergence of the eigenvalues to their
limit points  for rapidly decaying potential perturbations.
It was found that for a compactly supported
electric potential the eigenvalues converge to Landau
levels superexponentially fast. A similar effect was observed for the perturbation by
a compactly supported magnetic field \cite{RozTa}
or impenetrable compact
obstacle \cite{PuRoz}.
Another natural problem is to analyze the eigenvalue behavior  as
the coupling constant in front of the perturbation becomes large.
In this case the eigenvalue asymptotics  is described by semi-classical formulas,
as shown in \cite{HeHe}, \cite{HeLe}. We refer to the above references for further
bibliography.

Our objective is to study the discrete spectrum of the Landau Hamiltonian
$H_0$ perturbed by an expanding potential $V^{(t)}(\bx) = V(t^{-1}\bx), t >0$.
Under the condition $V\in \plainL1(\R^2)\cap \plainL2(\R^2)$ the
operator $H = H^{(t)} = H_0+V^{(t)}$,
is properly defined as an operator sum in $\plainL2(\R^2)$,
and $V^{(t)}$ is $H_0$-compact.
Our aim is to investigate the number  $N(\l_1, \l_2; H^{(t)})$ of
the eigenvalues of $H^{(t)}$
on the interval $(\l_1, \l_2)\Subset (\L_\nu, \L_{\nu+1})$
with some $\nu = -1,0, 1, \dots,$ as $t\to \infty$.
If $\l_1 = -\infty$, we write $N(\l_2; H^{(t)})$.

The behavior of $N(\l_1, \l_2; H^{(t)})$ is
 determined  by the potential $V$ as
follows. For any $V\in\plainL1(\R^2)$ and
$-\infty\le \l <\mu\le \infty$, we define
\begin{equation}\label{lm:eq}
A(\l, \mu; V) = |\{ \bx: \l < V(\bx) < \mu\}|,\
A^{(0)}(\l; V) = |\{\bx: V(\bx) = \l\}|.
\end{equation}
The coefficients $A(\l, \mu; V)$ are finite if $\l\mu>0$. Say, for
$\mu > \l > 0$, we obtain by the Chebyshev inequality
\begin{equation*}
A(\l, \mu; V) \le  A(\l, \infty; V) < \frac{1}{\l}\|V\|_1 < \infty.
\end{equation*}
The coefficient $A$ is monotone in $\l, \mu$, so that the limits
$A(\l\pm 0, \mu\pm 0; V)$ are well defined with various combinations of signs $\pm$.
We will call the number $\l$ a \emph{generic}
value for $V$ if $A^{(0)}(\l; V)=0$; otherwise, this value
is called \emph{exceptional}. For a given
function V, there are at most countably many exceptional
values. For a generic $\l$, $A(\l-0, \mu; V) = A^{(\pm)}(\l, \mu; V)$,
otherwise, $A(\l-0, \mu; V) =
A(\l, \mu; V) + A^{(0)}(\l; V)$, and similarly for $\mu$.

For any real $ \l_1 < \l_2$, such that  $[\l_1, \l_2]$
does not contain any of $\L_q$, we define
\begin{equation}\label{1:fullcoeff}
\COA(\l_1, \l_2; V) = \frac{B}{2\pi}\sum_{q=0}^\infty
A(\l_1 - \L_q, \l_2-\L_q; V).
\end{equation}
Since the sets $\{\bx: \L_q -
\l_2<|V(\bx)|< \L_q - \l_1\}$ are disjoint for different $q$'s, this series
converges for $V\in\plainL1(\R^2)$, and
\begin{equation*}
\COA(\l_1, \l_2; V) \le A\bigl(
\min\{|\L_q -
\l_2|,|\L_q-\l_1|\}, \infty; |V| \bigr)<\infty.
\end{equation*}
%The coefficient $\COA$ is monotone in $\l_1, \l_2$, so that the limit
%$\COA(\l_1-0, \l_2+0; V)$ is well defined.
%
The main result of the paper is contained in the
following Theorem:

\begin{thm}\label{main:thm}
Let $(\l_1, \l_2)\Subset (\L_\nu, \L_{\nu+1})$ with
some $\nu = -1,0, 1, \dots$. Suppose that
$V\in\plainL1(\R^2)\cap\plainL2(\R^2)$.
Then
\begin{equation*}%\label{MainFormula}
\begin{split}
    \COA(\l_1, \l_2; V)\le &\ \liminf_{t\to\infty}t^{-2}N(\l_1, \l_2; H^{(t)})\\[0.2cm]
    \le&\ \limsup_{t\to\infty}t^{-2}N(\l_1, \l_2; H^{(t)})
    \le \COA(\l_1-0, \l_2+0; V).
    \end{split}
\end{equation*}
If $\l_1-\L_q$ and $\l_2 - \L_q$ are generic for $V$ for all $q = 0, 1, \dots$,
then
%
%\begin{equation}\label{mainlevel:eq}
%A^{(0)}(\l_1-\L_q, V) = A^{(0)}(\l_2-\L_q, V) = 0,\ q
%= 0, 1, 2, \dots.
%\end{equation}
%Then
\begin{equation}\label{main:eq}
\lim_{t\to\infty} t^{-2} N(\l_1, \l_2; H^{(t)}) =
\COA(\l_1, \l_2; V).
\end{equation}
\end{thm}

Note that the right hand side of the asymptotic formula \eqref{main:eq}
coincides with the natural quasi-classical expression
for the counting function
of the magnetic Schr\"odinger operator, see e.g.
\cite{Sob}. On the other hand one might
juxtapose this result with the classical Szeg\"o Theorem
deriving the canonical distribution for the Toeplitz type operators of Fourier
type, see \cite{GreSze}, Theorem 8.6(c).
This comparison is even more appropriate since the proof of Theorem
\ref{main:thm} relies on the spectral analysis of the Toeplitz type operator
$T^{(t)} = P_q V^{(t)} P_q$, where $P_q$ is the projection on the spectral subspace
associated with the Landau level $\L_q$.
%In fact, the investigation of this operator is at the very
%heart of the matter.
Remembering that the non-zero spectra of the operators
$AB$ and $BA$ ($A, B$ being both compact) coincide,
instead of $T^{(t)}$ it is often more convenient to study the operator
$S^{(t)} = WP_q \overline W$ with a function $W$. If $W$ is radially symmetric, then
this operator splits into an orthogonal sum of one-dimensional operators,
whose eigenvalues (and their asymptotics) are computed using the explicit
formula for the integral kernel of $P_q$ (see \eqref{projection:eq}).
To handle the general case
we apply a method which is based on the approach put forward
by M. Birman and M. Solomyak to study weakly polar integral operators,
see \cite{BSpaper1}.
Precisely, we partition the plane $\R^2$
into disjoint annular sectors, i.e. domains of the form
$$
\Om_{m, l} = \biggl\{(\rho, \phi): (m-1)d<\rho\le md,\ \frac{2\pi}{N}(l-1)
< \phi\le \frac{2\pi}{N}l\biggr\},
m\in\mathbb N, l = 1, 2, \dots, N,
$$
with a fixed $d>0$ and natural $N$.
Choosing appropriate $d$ and $N$,
we approximate $W$ by a function which is constant on each $\Om_{m, l}$.
This reduces the operator $S^{(t)}$ to a block-matrix form.
The crucial point is that the off-diagonal entries do not contribute to the asymptotics,
which implies the "additivity" of the asymptotics in the function $W$. This property
allows one to reduce the problem to the radially symmetric case, for which
the eigenvalues are found in the closed form.

The above Toeplitz operators are linked with the initial Schr\"odinger operator
using the elementary formula
\begin{equation}\label{1:transf}
N(\l_1, \l_2;
H^{(t)})=N(-\infty,b^2; L^{(t)}), L^{(t)} = (H_0+V^{(t)}-a)^2,
\end{equation}
 where $a=(\l_1+\l_2)/2$, $b=(\l_2-\l_1)/2$, which was used
 previously in \cite{IT}, \cite{melroz03} in similar circumstances.
We represent $L^{(t)}$ in the block-matrix form with the entries
of the form $P_q L^{(t)} P_{q'}$. The
off-diagonal terms do not affect the asymptotics, and the diagonal
ones are directly expressed via the Toeplitz operators of the form $T^{(t)}$.

The paper is organized as follows. Section \ref{AsCoeff} is devoted to some abstract
operator theory. For families of semi-bounded operators depending
on a parameter $t\in(0,\infty)$ we introduce the
\emph{asymptotic coefficients} describing the
asymptotic distribution as $t\to \infty$ of
eigenvalues in a given interval, and establish their general properties
which are used throughout the the paper in different concrete environments.
In Sections
\ref{Bounds:Section} and \ref{Toeplitz:Section} we prove the crucial asymptotic formulas
for the Toeplitz operators $T^{(t)}$.
The reduction of the initial problem to the Toeplitz operators is implemented
in Sections \ref{reduction:sect} and
\ref{nonsmooth}.
The Appendix contains some elementary analytic
properties of level sets, needed for our proofs.

\textbf{Acknowledgements.}
The asymptotic problem considered in the paper appeared as a result of
discussions of the first author (GR) with
his colleague-physicist, R. Shekhter. A decisive part
of the work was done by the authors when enjoying  the
hospitality of the Isaac Newton Institute for
Mathematical Sciences (Cambridge, UK) in the framework
of the programs ``Spectral Theory and Partial
Differential Equations'' and ``Analysis on Graphs''.
It is our pleasure to thank the Institute and the
organizers of the programs for providing this
opportunity.

\section{Asymptotic coefficients}\label{AsCoeff}

Let $L$ be a self-adjoint operator, semi-bounded from below, with
\begin{equation*}
\eta_0 = \eta_0(L) = \inf \s_{\textup{\tiny ess}} (L).
\end{equation*}
Denote by $N(\eta; L), \eta < \eta_0$ the number of discrete eigenvalues of
$L$ strictly below $\eta$.
If $K$ is a compact operator, then $\s_{\textup{\tiny ess}} (K) = \{0\}$ and
we use the traditional notation
\begin{equation*}
n_-(\l; K) = N(-\l; K),\ n_+(\l; K) = N (-\l; -K), \l >0,
\end{equation*}
for the counting functions of the negative and positive eigenvalues respectively. For
an arbitrary
compact operator $K$, not necessarily self-adjoint,
introduce also the counting function
of its singular values:
\begin{equation*}
n(\l; K) = n_+(\l^2; K^*K).
\end{equation*}
Distribution functions for eigenvalues of the
sum (and difference) of operators satisfy certain
inequalities. For compact operators they are known as Ky
Fan inequalities and are presented in numerous
sources. However we need this kind of relations for general
semi-bounded operators. Although these inequalities
are by no means new, we were unable to locate a
convenient reference for the particular form we intend to use.
Therefore we present here a short explanation.

First, recall the version of the min-max principle for the counting function,
which is usually referred to as Glazman's Lemma:
%
%
%Glazman's lemma, a
%version of the minimax principle for eigenvalues of
%semi-bounded operators.

\begin{lem}\label{Glazman:lem}
Let $L$ be a semi-bounded operator with $\eta_0=\inf\s_{\textup{\tiny ess}}(L)$,
and let $\bd(L)$ be the domain of the quadratic form $l[u]$ of the operator  $L$.
Then
$\eta< \eta_0$ if and only if there exists a linear set
\footnote{Subspace, which is not necessarily closed}
$\COL\subset  \bd(L)$ of finite codimension, satisfying
the property
\begin{equation}\label{CL:eq}
l[u]\ge \eta\|u\|^2 \ \ \ \textup{for all }\ u\in \COL.
\end{equation}
Moreover,
\begin{equation}\label{Glazman}
N(\eta; L)=\min \codim\COL,
\end{equation}
where the minimum is taken over all linear sets
$\COL\subset \bd(L)$, satisfying \eqref{CL:eq}.
\end{lem}

In this form Glazman's Lemma appeared in \cite{Glaz}
(Ch. 1, Theorems 12, 12bis); equivalent formulations (however in
terms of eigenvalues, and not the distribution
function), are given in many  books on spectral
theory.

The most general form of the eigenvalue distribution function inequality we
need is the following.

\begin{lem} \label{PertLemma.Gen}
Let $L_1, L_2$ be semi-bounded (from below) self-adjoint
operators
such that $\bd(L_1)\subset \bd(L_2)$ and $L_2$ is $L_1$-form bounded
with a bound strictly less than $1$.
Then
$\eta_0(L_1+L_2)\ge \eta_0(L_1) + \eta_0(L_2)$
and
\begin{equation}\label{addgen1:eq}
N(\eta_1 + \eta_2; L_1+L_2) \le N(\eta_1;  L_1) + N(\eta_2; L_2)
\end{equation}
for any $\eta_j<\eta_0(L_j)$, $j=1,2$.
and
\begin{equation}\label{addgen2:eq}
N(\eta_1 - \eta_2; L_1 - L_2) \ge N(\eta_1; L_1) - N(\eta_2; L_2),
\end{equation}
for any $\eta_1, \eta_2$ such that
$\eta_1-\eta_2  <\eta_0(L_1-L_2)$, $\eta_2 < \eta_0(L_2)$.
\end{lem}

\begin{proof}
Denote $N_j = N(\eta_j; L_j)$, and let
$\COL_j\in\bd(L_j)$ be a subspace of codimension
$N_j$ on which $(L_ju,u)\ge\eta_j\|u\|^2$, $j = 1, 2$.
Then the subspace $\COL=\COL_1\cap\COL_2\subset \bd(L_1+L_2)$
has codimension not greater
than $N_1+N_2$, and  for $u\in\COL$,
\begin{equation*}
    ((L_1 + L_2)u,u)\ge (\eta_1+\eta_2)\|u\|^2,
\end{equation*}
so \eqref{addgen1:eq} follows from Glazman's lemma.
The inequality \eqref{addgen2:eq} follows from
\eqref{addgen1:eq} by
%moving the second term on the
%right in \eqref{addgen1:eq} to the left-hand side and
an obvious change of notation.
\end{proof}

We also need a simpler version of the above
inequalities. Let $L$ be a semi-bounded operator, and
let $K$ be compact and self-adjoint. Clearly,
$\eta_0(L + K) = \eta_0(L)$, and by
Lemma~\ref{PertLemma.Gen},
\begin{equation}\label{addgen:eq}
N(\eta - \l; L + K) \le N(\eta; L) + n_-(\l; K)
\end{equation}
for any $\eta<\eta_0$, $\l >0$.
It is useful to write a similar inequality for a pair of compact
self-adjoint operators $K_1, K_2$:
\begin{equation}\label{add:eq}
n_\pm(\l_1+\l_2; K_1 + K_2)\le n_\pm(\l_1; K_1) + n_\pm(\l_2; K_2),
\end{equation}
for any $\l_1, \l_2>0$.
For a pair of compact operators (not necessarily self-adjoint)
a similar inequality holds:
\begin{equation}\label{adds:eq}
n(\l_1+\l_2; K_1 + K_2)\le n(\l_1; K_1) + n(\l_2; K_2)
\end{equation}
for any $\l_1, \l_2>0$, see \cite{BS}, Section 9.2, Theorem 9.

For a family of semi-bounded operators $L = L^{(t)}$,
depending on a parameter $t>0$, with the value $\eta_0
= \eta_0(L^{(t)})$ independent of $t$, we introduce
the asymptotic coefficients
\begin{equation*}
\GB(\eta; L)  =  \limsup_{t\to\infty}
t^{-2} N\bigl(\eta; L^{(t)}\bigr),\ \
\gb(\eta;  L)  = \liminf_{t\to\infty}
t^{-2} N\bigl(\eta; L^{(t)}), \ \ \eta < \eta_0.
\end{equation*}
Clearly, one can introduce similar
asymptotic coefficients, with $t^{-2}$ replaced by
$t^{-\gamma}$ with any $\gamma > 0$.
Although such characteristics of operator families may prove to be useful
for some other eigenvalue counting
problems, the case $\g = 2$ is sufficient for our purposes.
General properties of such coefficients are the same as
for $\gamma=2$. The asymptotic coefficients, just
introduced, are not necessarily continuous in $\eta$, but they
are monotone.
We systematically use naturally defined limits such as $\GB(\eta\pm 0; L)$.

In order to keep in line with the traditional definition
of counting functions $n_\pm$ for compact operators,
for a compact self-adjoint family $K = K^{(t)}$ we denote
\begin{equation*}
\GR^{(\pm)}(\l; K) = \GB(-\l; \mp K),\ \gr^{(\pm)}(\l; K) = \gb(-\l; \mp K),
\end{equation*}
and for arbitrary compact family introduce also
\begin{equation*}
\GR (\l; K) = \GR^{(+)}(\l^2; K^*K),\ \gr(\l; K) = \gr^{(+)}(\l^2; K^*K).
\end{equation*}
The bounds \eqref{addgen1:eq},
\eqref{addgen2:eq}, \eqref{addgen:eq} and \eqref{add:eq}
imply similar bounds for the functionals $\GB, \gb$. In particular,
for a semi-bounded $L$ and compact $K$ it follows from \eqref{addgen:eq} that
\begin{equation}\label{gmadd:eq:K1}
\begin{split}
\GB(\eta-\l; L+K)\le & \ \GB(\eta; L) + \GR^{(-)}(\l; K),\\[0.2cm]
\gb(\eta-\l; L+K)\le & \ \gb(\eta; L) + \GR^{(-)}(\l; K),\ \ \eta < \eta_0,\  \l >0.
\end{split}
\end{equation}
For compact operators, the inequality \eqref{add:eq} produces the bounds
\begin{equation}\label{gmadd:eq:K}
\begin{split}
\GR(\l_1+\l_2; K_1+K_2) \le \GR(\l_1; K_1)
+ \GR(\l_2; K_2),\\[0.2cm]
\gr(\l_1+\l_2; K_1+K_2) \le \gr(\l_1; K_1) +
\GR(\l_2; K_2),
\end{split}
\end{equation}
for any $\l_1, \l_2>0$
and similar bounds hold for the functionals $\GR^{(\pm)}$, $\gr^{(\pm)}$.

We systematically use analogues of Birman-Solomyak
asymptotic perturbation lemma for the eigenvalues, see
\cite{BSpaper}.

\begin{lem}\label{AsLemma2}
Let $L=L^{(t)}, t>0,$ be a family of self-adjoint semi-bounded from below
operators with a value of $\eta_0 = \eta_0(L^{(t)})$ independent of $t$.
Suppose that for any $\d>0$ the family $L$ can be represented as
$L = L_{\d} + Y'_{\d} + Y''_{\d}$
with some $L^{(t)}$-form-compact self-adjoint
$Y'_{\d} = Y'^{(t)}_\d$ and
$Y''_{\d} = Y''^{(t)}_\d$, such that
\begin{equation}\label{AsLemma2:eq}
    \lim_{\d\to0}\GB(\tau;  L + M Y''_\d)=0
\end{equation}
for any $\tau < \eta_0$ and any $M\in \R$.
Then for any $\eta < \eta_0$
\begin{equation}\label{AsLemma2:eqM1}
\begin{split}
  \lim_{\e\downarrow 0} &\liminf_{M\uparrow 1}
  \liminf_{\d\to 0}\GB(\eta-\e; L_\d + MY'_{\d})
   \\
   &\  \le
   \GB(\eta; L)\le
   \lim_{\e\downarrow 0} \limsup_{M\downarrow 1}
   \limsup_{\d\to0}\GB(\eta+\e; L_\d+ MY'_{\d});\\
\lim_{\e\downarrow 0} &\liminf_{M\uparrow 1}
\liminf_{\d\to 0}\gb(\eta-\e;  L_\d + MY'_{\d})\\[0.2cm]
\le & \ \gb(\eta; L) \le
\lim_{\e\downarrow 0}\limsup_{M\downarrow 1}
\limsup_{\d\to0}\gb(\eta+\e; L_\d+ MY'_{\d}).
   \end{split}
\end{equation}
\end{lem}

\begin{proof}  It suffices to prove the Lemma for $\eta_0 = 0$.
Using \eqref{addgen1:eq} with $L_1 = (1-\mu) L + Y'_\d$ and
$L_2 = \mu L + Y''_\d$, we obtain
%It follows from \eqref{addgen1:eq} that
\begin{equation*}
\gb(\eta; L) \le \gb  (\eta + \e; (1-\mu) L_{\d} + Y'_\d)
+ \GB(-\e; \mu L_{\d} + Y''_{\d}),
\end{equation*}
for any $\eta <0$, $\mu\in (0, 1)$ and any $0<\e<|\eta| $.
Using \eqref{AsLemma2:eq}, and passing to the limit as $\d\to 0$, we get
\begin{equation*}
\gb(\eta; L)
\le   \limsup_{\d\to 0}
\gb(\eta+\e;  L_{\d} + (1-\mu)^{-1}Y'_\d).
\end{equation*}
Passing to the limit as $\mu\downarrow 0$ and  $\e\downarrow 0$,
we get the proclaimed upper bound for $\gb(\eta; L)$, where
$M=(1-\mu)^{-1}$.  Similarly for $\GB(\eta; L)$.

For the lower bound we use \eqref{addgen2:eq}
with $L_1 = (1+\mu) L + Y'_\d$ and
$L_2 = \mu L - Y''_\d$, which gives
\begin{equation*}
\gb(\eta; L)\ge \gb(\eta-\e; (1+\mu)L_{\d} + Y'_\d)
- \GB(-\e; \mu L_{\d} - Y''_\d)
\end{equation*}
for any $\eta <0$, $\mu\in (0, 1)$ and $\e>0$.
Using \eqref{AsLemma2:eq} again, and passing to the limit as $\d\to 0$, we get
\begin{equation*}
\gb(\eta; L)
\ge   \liminf_{\d\to 0}
\gb( \eta-\e;  L_{\d} + (1+\mu)^{-1}Y'_\d).
\end{equation*}
Passing to the limit as $\mu\downarrow 0$ and  $\e\downarrow 0$,
we get the proclaimed lower bound for $\gb(\eta; L)$,
where $M = (1+\mu)^{-1}$.
Similarly for $\GB(\eta; L)$.
\end{proof}

\begin{lem}\label{AsLemma1}
Let $L=L^{(t)}, t>0,$ be a family of self-adjoint semi-bounded from below
operators with a value of $\eta_0 = \eta_0(L^{(t)})$ independent of $t$.
Suppose that for any $\d>0$ the family $L$ can be represented as
$L = L_{\d} + K'_{\d}$ with a compact self-adjoint $K'_{\d} = K'^{(t)}_\d$, such that
\begin{equation}\label{AsLemma:eq}
    \lim_{\d\to0}\GR(\e; K'_\d)=0
\end{equation}
for any $\e>0$.
Then for any $\eta < \eta_0$
\begin{equation}\label{AsLemma:eqM1}
\begin{split}
   &\lim_{\e\downarrow0}\liminf_{\d\to 0}\GB(\eta-\e; L_\d)\le
   \GB(\eta; L)\le\lim_{\e\downarrow0}\limsup_{\d\to0}\GB(\eta+\e; L_\d);\\
&\lim_{\e\downarrow0}\liminf_{\d\to0}\gb(\eta-\e; L_\d)\le
   \gb(\eta; L)\le\lim_{\e\downarrow0}\limsup_{\d\to 0}\gb(\eta + \e; L_\d).
   \end{split}
\end{equation}
\end{lem}

\begin{proof}For $\d$ fixed, write \eqref{gmadd:eq:K1}
for families $L_\d, K'_\d$
and pass to $\limsup$ as $\d\to 0$ and then
as $\e\to0$. This proves the upper bounds in \eqref{AsLemma:eqM1}.
The lower bounds are proved
similarly.
\end{proof}

The next result is a direct consequence of this lemma applied to compact operators:

\begin{lem}\label{AsLemma}
Let $K=K^{(t)}$ be a family of compact
operators. Suppose that for any $\d>0$ the family $K$
can be represented as a sum $K=K_\d+K'_\d$ such that
for any $\e>0$ the condition \eqref{AsLemma:eq} is satisfied.
Then for any $\l>0$
\begin{equation}\label{AsLemma:eqM}
\begin{split}
   &\lim_{\e\downarrow0}\liminf_{\d\to0}\GR(\l+\e;K_\d)\le
   \GR(\l,K)\le\lim_{\e\downarrow0}\limsup_{\d\to0}\GR(\l-\e;K_\d);\\
&\lim_{\e\downarrow0}\liminf_{\d\to0}\gr(\l+\e;K_\d)\le
   \gr(\l,K)\le\lim_{\e\downarrow0}\limsup_{\d\to0}\gr(\l-\e;K_\d).
   \end{split}
\end{equation}
If, moreover, the families $K,K_\d,K'_\d$ are
self-adjoint, then the relations \eqref{AsLemma:eqM}
hold with $\GR$, $\gr$ replaced respectively by
$\GR^{(\pm)}, \gr^{(\pm)}$.
\end{lem}

\section{Eigenvalue bounds for Toeplitz operators}\label{Bounds:Section}

\subsection{Eigenvalue bounds for auxiliary integral operators}
Here we obtain spectral estimates for integral operators involving
the projections $P_q$ on the spectral subspaces
 (\emph{Landau
subspaces}) $\CL_q$ associated with
the landau Levels $\L_q$, $q = 0, 1, 2, \dots$.
Choosing the gauge $\ba=(-\frac{B}{2}x_2,\frac{B}{2}x_1)$
for the magnetic potential, one can write
the orthonormal basis of the subspace $\CL_q$
using the generalized Laguerre polynomials
\begin{equation*}
L_q^{(\a)}(\xi) = \sum_{m=0}^q \begin{pmatrix} q+\a \\
q - m\end{pmatrix} \frac{(-\xi)^m}{m!}, \ \xi\ge 0,
\end{equation*}
as follows:
\begin{equation}\label{psi:eq}
\psi_{q, \a}(\bx) = \sqrt{\frac{q!}{(q+\a)!}}
\biggl[\sqrt{\frac{B}{2}}(x_1+i x_2)\biggr]^\a
L_q^{(\a)}\biggl(\frac{B|\bx|^2}{2}\biggr)
\sqrt{\frac{B}{2\pi}}\exp{\biggl(-\frac{B|\bx|^2}{4}\biggr)},
\end{equation}
 for $\a=-q,-q+1,\dots.$ The orthonormality follows from the standard relation
\begin{equation}\label{normalise:eq}
\int_0^\infty  \xi^\a e^{-\xi}L_q^{(\a)}(\xi)
L_{q'}^{(\a)}(\xi)d\xi = \frac{\G(\a+q+1)}{q!}\d_{q,
q'}.
\end{equation}
 The integral kernel of the projection $P_q$ is
\begin{equation}\label{projection:eq}
\CP_q(\bx, \by) = \frac{B}{2\pi}
L_q^{(0)}\biggl(\frac{B|\bx-\by|^2}{2}\biggr)
\exp\biggl(-\frac{B}{4}\bigl( |\bx-\by|^2 + 2i \bx
\wedge\by\bigr) \biggr).
\end{equation}
 The following important estimate for the  Laguerre
 polynomials can be found in \cite{RaiWar}.

\begin{lem} Let $k\in\mathbb Z_+$. Then
\begin{equation}\label{ocenka:eq}
|L_k^{(\a)}(\xi)|\le (\a + k)^k e^{\frac{\xi}{\a+k}}
\end{equation}
for all $\xi\ge 0$ and $\a\ge 1-k$.
\end{lem}

For $t>0$ and any function $f = f(\bx)$ we denote $f^{(t)}(\bx) =
f(t^{-1} \bx)$. We consider the operator families of the form
\begin{equation*}
S^{(t)}(W_1, W_2) = S_q^{(t)}(W_1, W_2) = W_1^{(t)} P_q
\overline{W_2^{(t)}}, t >  0,
\end{equation*}
where $W_1, W_2$ are some complex-valued functions. Along with
$S^{(t)}$ we also consider
\begin{equation*}
T_{q, q'}^{(t)}(V) = P_q V^{(t)} P_{q'},\ T_q^{(t)}(V) = T_{q,
q}^{(t)}(V),
\end{equation*}
with some complex-valued function $V$; these are Toeplitz type
operators for $q'=q$ and Hankel type operators for $q'\ne q$. The
labels $q, q' = 0,1, 2, \dots $ are fixed and as a rule, they are
not reflected in the notation of the operators. The superscript
${}^{(t)}$ is sometimes omitted as well. The functions $W_j,V$ will
be referred to as weight functions. It is convenient to represent
$S^{(t)}(W_1, W_2)$ and $T^{(t)}(V)$ in terms of the operator
$Z^{(t)}(W) = W^{(t)} P_q$, so that
\begin{gather*}
S^{(t)}(W_1, W_2) = Z^{(t)}(W_1) (Z^{(t)}(W_2))^*,\\[0.2cm]
T^{(t)}(V) = (Z^{(t)}(V_1))^*(Z^{(t)}(V_2)), \ V_1 = \sqrt{|V|},\
V_2 = V |V|^{-1/2}.
\end{gather*}
Under mild assumptions on $W_1, W_2, V$ the above operators are
compact.

\begin{lem}\label{trace:lem}
If $W, W_1, W_2\in \plainL2(\R^2)$,  $\;V\in \plainL1(\R^2)$, then
$Z^{(t)}(W)\in\GS_2$, $S^{(t)}(W_1,$$ W_2)$, $T^{(t)}(V)\in\GS_1$,
and
\begin{gather*}
\|Z^{(t)}(W)\|_{\GS_2} = t \sqrt{\frac{B}{2\pi}}\|W\|_2,\\[0.2cm]
\| S^{(t)} (W_1, W_2)\|_{\GS_1}\le t^2 \frac{B}{2\pi}\|W_1\|_2
\|W_2\|_2,\ \ \ \| T^{(t)}(V)\|_{\GS_1}\le t^2
\frac{B}{2\pi}\|V\|_1.
\end{gather*}
\end{lem}

\begin{proof}
It suffices to prove the equality for the Hilbert-Schmidt norm of
$Z^{(t)}(W)$. Using \eqref{projection:eq} and \eqref{normalise:eq},
we find
\begin{align*}
\|Z^{(t)}(W)\|_{\GS_2}^2 = &\ \int_{\R^2} \int_{\R^2}
|W(t^{-1}\bx)|^2 |\CP_q (\bx, \by)|^2 d\bx d\by\\[0.2cm]
= &\ \frac{B^2}{(2\pi)^2}\int_{\R^2}\int_{\R^2} |W(t^{-1}\bx)|^2
\biggl(L_q^{(0)}\biggl(\frac{B|\by|^2}{2}\biggr)\biggr)^2
\exp\biggl(-\frac{B}{2}|\by|^2\biggr) d\bx d\by\\[0.2cm]
= &\ t^2 \|W\|_2^2 \frac{B}{2\pi}\int_0^\infty
\bigl(L_q^{(0)}(s)\bigr)^2 e^{-s} ds = t^2 \|W\|_2^2 \frac{B}{2\pi},
\end{align*}
as required.
\end{proof}

Using the notations introduced in Section~\ref{AsCoeff}, we set
\begin{equation}\label{McoeffDef}
\GM(\l; W_1, W_2) = \GR(\l; S^{(t)} (W_1, W_2));\ \gm(\l; W_1, W_2)
= \gr(\l; S^{(t)}(W_1, W_2)\bigr).
\end{equation}
When $W_1 = W_2 = W$, we write $\GM(\l; W)$ and $\gm(\l; W)$.  In
this case the operator $S(W_1, W_2)$ is self-adjoint, so that we can
also define the functionals $\GM^{(\pm)}(W)$. Clearly, $\GM^{(-)}(W)
= 0$ and $\GM^{(+)}(W) = \GM(W)$.

For the operator $T^{(t)} = T^{(t)}_{q, q'}$ we introduce the
related quantities:
\begin{equation*}
\GN(\l; V)  = \GR\bigl(\l; T^{(t)}(V)\bigr),\ \gn(\l; V) =
\gr\bigl(\l; T^{(t)}(V)\bigr),
\end{equation*}
and in case  when $V$ is real-valued and $q = q'$, we introduce the
natural notation $\GN^{(\pm)}(V)$ and $\gn^{(\pm)}(V)$ as well.
Since the nonzero eigenvalues of $S^{(t)}_q(W)$ and
$T^{(t)}_q(|W|^2)$ coincide, we have
\begin{equation}\label{mn:eq}
\GN^{(+)}(\l; |W|^2) = \GM(\l; W), \ \ \gn^{(+)}(\l; |W|^2) =
\gm(\l; W).
\end{equation}
If necessary, we reflect the dependence on  $q, q'$ in the notation
of the above asymptotic coefficients: for instance, we may write
$\gm_{q}(W_1, W_2)$ and $\GN^{(\pm)}_{q}(V)$, $\GN_{q, q'}(V)$.

Now Lemma \ref{trace:lem} leads to the following result.

\begin{lem}\label{base:lem}
If $W, W_1, W_2\in \plainL2(\R^2)$ and $V\in\plainL1(\R^2)$, then
for any $\l >0$ and $q, q' = 0, 1, 2, \dots$, we have
\begin{gather*}%\label{MN:eq}
\GR(\l; Z_q^{(t)}(W))\le \dfrac{B}{2\pi \l^2} \|W\|_2^2,\\[0.2cm]
\GM_q(\l; W_1, W_2)\le \dfrac{B}{2\pi\l}\| W_1\|_2 \|W_2\|_2,\ \
\GN_{q, q'}(\l; V)\le \dfrac{B}{2\pi\l}\|V\|_1.
\end{gather*}
\end{lem}

\begin{proof}
For any operator $T\in \GS_p$ we have $n(\l; T)\le \l^{-p}
\|T\|_{\GS_p}^p$. It remains to use Lemma \ref{trace:lem}.
\end{proof}

\subsection{Localization} In our study
of the eigenvalue behavior, we systematically represent the
operators as block-matrices associated with certain orthogonal
decompositions. The results of this subsection help to show that
off-diagonal terms do not contribute to the asymptotic coefficients.

We consider an auxiliary integral operator $K^{(t)}(W_1, W_2, f)$
having the kernel
\begin{equation*}
\CK^{(t)}_q(\bx, \by; W_1, W_2, f) = W_1^{(t)}(\bx) \CP_q(\bx, \by)
f^{(t)}(\bx-\by)W_2^{(t)}(\by),
\end{equation*}
with some functions $W_1, W_2$ and $f$.

\begin{lem}\label{disjoint:lem}
Let $W_1\in \plainL{p}, W_2\in \plainL{s}, f\in \plainL{\infty}$
with arbitrary $2\le p, s\le \infty$ such that $ \frac{1}{p} +
\frac{1}{s}\ge \frac{1}{2}.$ Suppose that for some $\d >0$
\begin{equation*}
f(\bz) = 0, \ \ \textup{for}\ \ \ |\bz|\le \d.
\end{equation*}
Then for any $\l >0$
\begin{equation}\label{Kzero:eq}
\GR\bigl(\l;\  K^{(t)}(W_1, W_2, f)\bigr) = 0.
\end{equation}
\end{lem}

\begin{proof} By assumption $\CK^{(t)}(\bx, \by; W_1,W_2, f) = 0$
if $|\bx-\by|\le \d t$. For $|\bx-\by|> \d t$ we use
\eqref{projection:eq}:
\begin{equation*}
|\CP_q(\bx, \by)|\le \frac{B}{2\pi} \biggl|L_q^{(0)}\biggl(
\frac{B|\bx-\by|^2}{2} \biggr)\biggr| e^{-\frac{B}{8}\d^2 t^2}
e^{-\frac{B}{8}  |\bx-\by|^2}.
\end{equation*}
Let $r\in [2, \infty]$ be defined by $p^{-1}+s^{-1}+r^{-1}=1$. By
the H\"older and Young inequalities
\begin{align*}
\| K^{(t)}(W_1, W_2, f)\|_{\GS_2}^2\le
 &\
  t^{\frac{4}{p} + \frac{4}{s}}\| W_1\|_p^2
\|W_2\|_s^2 \|f\|_\infty^2 \frac{B^2}{(2\pi)^2}
e^{-\frac{B}{4} \d^2 t^2}\\
&\ \times\biggl(\int_{\R^2}\biggl|L_q^{(0)}\biggl(
\frac{B|\bx|^2}{2}
\biggr)\biggr|^r e^{-\frac{B r}{4} |\bx|^2}  d\bx\biggr)^{\frac{2}{r}}\\
\le t^{\frac{4}{p} + \frac{4}{s}}& \| W_1\|_p^2 \|W_2\|_s^2
\|f\|_\infty^2 \biggl(\frac{B}{2\pi}\biggr)^{2-\frac{2}{r}}
e^{-\frac{B}{4} \d^2 t^2} \biggl(\int_0^\infty |L_q^{(0)}(\xi)|^r
e^{- r\xi/2}
d\xi\biggr)^{\frac{2}{r}}\\[0.2cm]
 \le &\  C t^{\frac{4}{p} + \frac{4}{s}} e^{-\frac{B}{4} \d^2
t^2}.
\end{align*}
%\end{gather*}
Due to the presence of the exponentially decaying factor, for
sufficiently large $t$  the operator $K^{(t)}(W_1,W_2, f)$ has no
singular values above $\l$, whence \eqref{Kzero:eq}.
\end{proof}

The above Lemma has a few useful corollaries.
%Let us establish a simple estimate for the operator $S^{(t)}(W_1, W_2)$ assuming
%that $W_1$ and $W_2$ have disjoint supports.

\begin{cor}\label{disjoint2:cor}
Let $V\in\plainL2(\R^2)$ and $R^{(t)}(V) = R_q^{(t)}(V) =
[P_q,V^{(t)}]$. Then
\begin{equation}\label{Nzero1:eq}
    \GR(\l, R^{(t)}(V))=0.
\end{equation}

Moreover, if $V\in \plainL1(\R^2) + \plainL2(\R^2)$ and $q\not = q'$
then also
\begin{equation}\label{Nzero:eq}
\GN_{q, q'}(\l; V) = 0,
\end{equation}
for all $\l >0$.
\end{cor}

\begin{proof}
Fix a $\d >0$ and find $\tilde V\in \plainC\infty_0(\R^2)$ such that
$\|V-\tilde V\|_2< \d$. We have
\begin{equation}\label{Nzero2:eq}
[P_q,V^{(t)}]=[P_q,\tilde V^{(t)}]+P_q(V^{(t)} - \tilde
V^{(t)})-(V^{(t)} - \tilde V^{(t)})P_q .\end{equation}
 To two last
terms in\eqref{Nzero2:eq} we can apply Lemma~\ref{base:lem}, which
gives
\begin{equation}\label{Nzero3:eq}
\GR(\e;P(V^{(t)}-\tilde V^{(t)})-(V^{(t)}-\tilde V^{(t)})P)\le
C\d^2/\e^2,
\end{equation}
for any $\e>0$. Since $\d$ is arbitrarily small, it suffices to
prove that $\GR(\l; [\tilde V^{(t)},P])=0$ for any $\l>0$ and
$\tilde V\in \plainC\infty_0(\R^2)$, and then  apply
Lemma~\ref{AsLemma}.
 The integral kernel of $[\tilde V^{(t)},P_q]$ is
\begin{equation*}
\bigl(\tilde V(t^{-1}\bx) - \tilde V(t^{-1}\by)\bigr)\CP_q(\bx,
\by).
\end{equation*}
Denoting $ f(\bz) = \chi(|\bz|\le  \e), $ we rewrite this kernel as
follows:
\begin{align}
\bigl(\tilde V(t^{-1}\bx) - \tilde V(t^{-1}\by)\bigr)&\ \CP_q(\bx,
\by)
f^{(t)}(\bx-\by)\label{split:eq}\\[0.2cm]
+ &\ \CK^{(t)}(\bx; \by; \tilde V, 1, 1-f) -  \CK^{(t)}(\bx; \by; 1,
\tilde V, 1-f).\notag
\end{align}
The operators, corresponding to the last two terms satisfy
\eqref{Kzero:eq}. For the first term in \eqref{split:eq}, we use
that $\tilde V$ has a compact support and so
\begin{equation*}
\bigl|\tilde V(t^{-1}\bx) - \tilde V(t^{-1}\by)\bigr|
f^{(t)}(\bx-\by) \le  t^{-1}\max_{\bz}|\nabla \tilde V(\bz)|\
|\bx-\by|\le C\e.
\end{equation*}
Thus the norm of the operator corresponding to the first term in
\eqref{split:eq} is bounded by
\begin{equation*}
C\e \max_{\bx }\int_{\R^2} |\CP_q(\bx, \by)| d\by\le C' \e,
\end{equation*}
and hence it can be made arbitrarily small, which proves
\eqref{Nzero1:eq}.

If $V\in \plainL2(\R^2)$, then \eqref{Nzero:eq} follows immediately
from \eqref{Nzero1:eq} in view of the identity $T_{q, q'}^{(t)}(V) =
R_q^{(t)}(V) P_{q'}$. If $V\in\plainL1(\R^2)$, then for arbitrary
$\d>0$ we approximate $V$ with a function $\tilde
V\in\plainC\infty_0(\R^2)$, such that $\|V-\tilde V\|_1 < \d$, and
use again Lemmas \ref{base:lem}, \ref{AsLemma} and the formula
\eqref{Nzero1:eq}.
\end{proof}

\begin{cor}\label{disjoint3:cor}
 Let $W_1\in\plainL2(\R^2)$ and $W_2\in\plainL\infty(\R^2)$ be such that
$W_1 W_2 = 0$. Then $\GM(\l; W_1, W_2) = 0$ for all $\l >0$.
\end{cor}

\begin{proof}
Rewrite: $S^{(t)}(W_1, W_2) = - R^{(t)}(W_1) W_2^{(t)}$.
Consequently,
\begin{equation*}
\GM(\l; W_1, W_2)\le \GR(\l \|W_2\|_{\infty}^{-1}; R^{(t)}(W_1) ) =
0,
\end{equation*}
by Corollary \ref{disjoint2:cor}.
\end{proof}

\section{Eigenvalues of Toeplitz operators}\label{Toeplitz:Section}

\subsection{Additivity of asymptotic coefficients}\label{add.subsect}
 Further on, we  will approximate  weight functions by piece-wise
constant ones. To describe these approximations we cut
the plane $\R^2$ in the following way. For fixed
 $N\in \mathbb N$, $d >0$ we tile $\R^2$ by  disjoint annular sectors
\begin{equation*}%\label{annulSector}
\Om_{m, l} = \biggl \{\bx=(\rho,\t) : (m-1)d< \rho \le
md , \frac{2\pi}{N}(l-1) < \t \le
\frac{2\pi}{N}l\biggr\},\ m\in\mathbb N, \ l = 1, 2,
\dots, N.
\end{equation*}
For any set $\Om$ we denote by $\chi(\bx\in \Om)$ its
characteristic function.
Let $X_{m, l} = \chi(\bx\in\Om_{m, l})$.

\begin{lem}\label{tiles:lem}
If $(m, l)\not = (m', l')$, then $\GM(\l; X_{m, l},
X_{m', l'}) = 0$ for all $\l >0$.
\end{lem}

\begin{proof} Immediately follows from
Corollary \ref{disjoint3:cor}.
\end{proof}

This result leads to the additivity of the asymptotic
coefficients for piece-wise constant functions of the
form
\begin{equation}\label{stepweight}
    W = \sum_{m, l} w_{m, l} X_{m, l},
\end{equation}
where $w_{m, l}$ are some complex numbers and the sum
is finite.

\begin{lem}\label{additivity.lem}
Let $W$ have the form \eqref{stepweight}. Then
\begin{equation}\label{add}
\sum \gm(\l+0;w_{m, l} X_{m, l})
\le \gm(\l; W)\le \GM(\l; W)\le \sum \GM(\l-0;w_{m, l} X_{m, l}),
\end{equation}
\end{lem}

\begin{proof}
We prove the upper bound only, the lower bound is established in the
same way, with obvious changes.
Represent the operator $S^{(t)}(W)$ as
\begin{align} \label{add.1}
S^{(t)}(W) = &\ \sum_{m,l}|w_{m,l}|^2 S^{(t)}(X_{m,l})
+\sum_{(m,l)\ne(m',l')}w_{m,l}\overline{w_{m',l'}}
S^{(t)}(X_{m,l},X_{m',l'})\notag\\[0.2cm]
= &\ S'+S''.
\end{align}
The family $S''$ in \eqref{add.1} is a finite sum of
operators of the form considered in
Lemma~\ref{tiles:lem}, therefore $\GR(\e, S'')=0$ for
any $\e>0$. Further, the operator $S'$ is a direct sum
of the operators $|w_{m,l}|^2 S^{(t)}(X_{m,l})$, therefore
its spectrum is the union of spectra of summands, so
\begin{equation}\label{add.3}
\GR(\l,S')\le \sum\GR(\l, |w_{m,l}|^2 S^{(t)}(X_{m,l})).
\end{equation}
Now we can apply Lemma~\ref{AsLemma}.
\end{proof}

Let us establish a similar additivity property for the
operator $T(V)$ with a real-valued function $V$ of the form
\begin{equation}\label{stepweight1}
V = \sum_{m, l} v_{m, l} X_{m, l},
\end{equation}
where $v_{m, l}$ are real and the sum is finite.

\begin{lem}\label{additivity1:lem}
Let $V$ be of the form \eqref{stepweight1}.
Then
\begin{equation}\label{addV}
\begin{split}
\sum_{\pm v_{m, l}>0} \gn^{(+)}(\l+0; \pm v_{m, l} X_{m, l})& \
\le\gn^{(\pm)}(\l; V) \\[0.2cm]
\le \GN^{(\pm)}(\l; V)\le &\
\sum_{\pm v_{m, l}>0} \GN^{(+)}(\l-0; \pm v_{m, l} X_{m, l}).
\end{split}
\end{equation}

\end{lem}

\begin{proof} As in the previous lemma, we prove only the upper bound.
Let $\Om = \cup \Om_{m, l}$ be the set where $V(\bx)\not = 0$, and let
$\Om_{0, 0} = \R^2\setminus \Om$. It is convenient to include the set $\Om_{0, 0}$ in the
family of $\Om_{m, l}$'s.
Rewrite:
\begin{equation*}
T^{(t)}(V) = \sum  X_{m, l}  T^{(t)}(v_{n, s}X_{n,s}) X_{m', l'} = T' + T'',
\end{equation*}
where
\begin{equation*}
T' = \sum  T_{m, l},\ \ T_{m, l} = X_{m, l} T^{(t)}(v_{m, l}X_{m,l}) X_{m, l},
\end{equation*}
and $T''$ is the sum in which at least one of the pairs
$(m, l), (m', l')$ is distinct from $(n, s)$. Consider, for instance the term with
$(m, l)\not = (n, s)$, and rewrite it as follows:
\begin{equation*}
X_{m, l}  T^{(t)}(v_{n, s})X_{m', l'} = v_{n, s} S^{(t)}(X_{m, l}, X_{n, s}) PX_{m', l'},
\end{equation*}
so that by Corollary \ref{disjoint3:cor}, the value $\GR(\e; \ \cdot\ )$ for this operator equals
zero for any $\e > 0$. Consequently, $\GR(\e; T'') = 0$ for any $\e>0$. Next, $T'$ is an orthogonal
sum of  operators $T_{m, l}$,  therefore
$$
n_{\pm}(\l; T') = \sum_{\pm v_{ml}>0}n_+(\l; \pm T_{m, l}),\
\GR^{(\pm)}(\l; T')\le \sum_{\pm v_{ml}>0} \GR^{(+)}(\l; \pm T_{m, l}).
$$
So, by  Lemma \ref{AsLemma},
\begin{equation*}
\GN^{(\pm)}(\l; V)\le \sum_{\pm v_{ml}>0} \GR^{(+)}(\l-0; \pm T_{m, l}).
\end{equation*}
Now we apply again Corollary~\ref{disjoint3:cor} and Lemma \ref{AsLemma} to each of operators
$T_{m,l}$, which gives
\begin{equation*}
\GR^{(+)}(\l-0; \pm T_{m, l})\le \GN^{(+)}(\l-0; \pm v_{m, l} X_{m, l}),
\end{equation*}
and this leads to the required upper bound.
\end{proof}

Another kind of additivity holds with respect to the
Landau projections. For $J>1$ we denote by $P^{(J)}$
the projection $P^{(J)}=\sum_{q\le J}P_q$. Consider
the Toeplitz family
$T^{(t)}_{(J)}(V)=P^{(J)}V^{(t)}P^{(J)}$.

\begin{lem}\label{AdditP:lem}
For $V\in \plainL1(\R^2) + \plainL2(\R^2)$, $\l>0,$
\begin{equation}\label{AdditP}
    \sum_{q\le J}\gn_q^{(\pm)}(\l+0; V)
    \le\gr^{(\pm)}\bigl(\l; T^{(t)}_{(J)}\bigr)
    \le\GR^{(\pm)}\bigl(\l; T^{(t)}_{(J)}\bigr)
    \le \sum_{q\le J}\GN_q^{(\pm)}(\l-0; V).
\end{equation}
\end{lem}

\begin{proof}
We split the family $T^{(t)}_{(J)}$ as follows:
\begin{equation*}
T^{(t)}_{(J)}=\sum_{q}T^{(t)}_{q}+\sum_{q\ne
q'}T^{(t)}_{q,q'}=T'+T''.
\end{equation*}
 The operators in $T'$ act in orthogonal
subspaces, so their distribution functions add up. By
Corollary~\ref{disjoint2:cor}
$\GR(\e; T^{(t)}_{q,q'})=0,\ q\ne q'$ for any $\e>0$, so that
$\GR(\e; T'') = 0$, and \eqref{AdditP}
follows by Lemma~\ref{AsLemma}.
\end{proof}

\subsection{Model integral operators} In order to  pass from the
conditional results in Subsection~\ref{add.subsect} to
actual calculations, we need at least some operators
for which the asymptotic coefficients are known. Here
we consider such model operators.

\begin{lem} Let $0\le d_1 < d_2<\infty$, and let $W = \chi(d_1< |\bx|< d_2)$.
Then for any $\l >0$
\begin{equation}\label{gm2:eq}
\GM(\l; W) = \gm(\l; W) =
\begin{cases}
\dfrac{B}{2}(d_2^2-d_1^2),\ \  &\l<1,\\[0.2cm]
0,\ \  &\l\ge 1.
\end{cases}
\end{equation}
\end{lem}

\begin{proof}
Since the function $W$ is radially symmetric, using \eqref{psi:eq}
we can immediately find all
eigenvalues of the operator $S_q^{(t)}(W)$ explicitly
(see \cite{RaiWar}, Lemma 3.4):
%\begin{align*}
\begin{equation}
\l_j = \l_j^{(t)} = \l_j^{(t)}(d_1, d_2)=
% &\
\int\limits_{d_1 t<|\bx| < d_2t} |\psi_{q, j}(\bx)|^2
d\bx
%\\[0.2cm]= &\
=\frac{q!}{(j+q)!} \int\limits_{\eta_1}^{\eta_2}
\xi^j e^{-\xi} \bigl(L_q^{(j)}(\xi)\bigr)^2 d\xi,
%\end{align*}
\end{equation}
where we denote $\eta_k=B(d_k t)^2/2, k=1,2$. Note that $\l_j$'s
are not necessarily labeled in the usual decreasing order.
The functions $\psi_{q, j}$ are normalized, so
$\l_j\le 1$ and therefore \eqref{gm2:eq} holds for $\l\ge 1$.

 Let now $\l <1$. We will find the asymptotics
 of $\l_j = \l_j^{(t)}(0, d)$ as $t\to\infty$.
Denote $\eta = B(td)^2/2$ and fix an $\e >0$.
Suppose first that $j\ge (1+\e)\eta$. Then
\eqref{ocenka:eq} implies
\begin{equation*}
\l_j\le (j+q)^{2q} \frac{q!}{(j+q)!} \int_0^\eta \xi^j
\exp\biggl(-\bigl(1-\frac{2}{j+q}\bigr)\xi\biggr) d\xi
\le (j+q)^{2q} \frac{q!}{(j+q)!} e^{\frac{2}{1+\e}}
\int_0^\eta \xi^j e^{-\xi} d\xi.
\end{equation*}
The maximum of the integrand is attained at $\xi = j$,
it grows for $\xi<j$, thus we can estimate it from
above by $\eta^j e^{-\eta}$, which leads to the bound
\begin{equation*}
\l_j\le C(j+q)^{2q} \frac{q!}{(j+q)!}
\eta^{j+1}e^{-\eta}.
\end{equation*}
Using the Stirling formula, we get
\begin{equation}\label{aj:eq}
\l_j \le  C \frac{(j+q)^{2q}q!}{(j+q)^{j+q + 1/2}} \eta^{j+1} e^{j+q - \eta} \le e^q
(j+q)^{q-1/2}q!
\eta \tau_j(\eta),\ \ \tau_j(\eta) = \biggl(\frac{\eta}{j}\biggr)^{j} e^{j - \eta}.
\end{equation}
 To estimate $\tau_j(\eta)$ we rewrite it as
\begin{equation*}
\tau_j(\eta) = \exp\biggl[ - \int_\eta^j
\biggl(\frac{j}{s}-1\biggr) ds \biggr].
\end{equation*}
Now we fix $\e_1\in (0, \e)$ and obtain:
\begin{align*}
%\begin{equation*}
\tau_j(\eta) \le   &\ \exp\biggl[ -  \!\!\! \!\!\!\int\limits_\eta^{j(1+\e_1)^{-1}}\!\!
\biggl(\frac{j}{s}-1\biggr)ds \biggr] \le \exp\biggl[
- \e_1\!\!\!\!\!\! \!\!\!\int
\limits_\eta^{j(1+\e_1)^{-1}} \!\!\! \!\!\!\!\!\!ds \biggr]\\[0.2cm]
\le &\ \exp\biggl( -j \e_1\frac{\e-\e_1}{(1+\e)(1+\e_1)}\biggr).
\end{align*}
%\end{equation*}
This shows that $\l_j$ tends to zero very fast as
$\eta\to\infty$ and $j\ge (1+\e)\eta$. Assume now that
$j\le (1-\e)\eta$. Since the functions $\psi_{q, j}$
are normalized, we have
\begin{equation*}
\l_j = 1 - \mu_j,\ \ \mu_j = \frac{q!}{(j+q)!}
\int_\eta^\infty \xi^j e^{-\xi}
\bigl(L_q^{(j)}(\xi)\bigr)^2 d\xi > 0.
\end{equation*}
Then, using \eqref{ocenka:eq} again, we obtain
\begin{equation*}
\mu_j\le (j+q)^{2q} \frac{q!}{(j+q)!} \int_\eta^\infty
\xi^j
\exp\biggl(-\bigl(1-\frac{2}{j+q}\bigr)\xi\biggr) d\xi
\end{equation*}
For an arbitrary $\e_1\in (0, 1)$ rewrite the
integrand as follows:
\begin{equation*}
\biggl[\xi^j e^{-(1-\e_1)\xi}\biggr]
\exp\biggl(-\bigl(\e_1-\frac{2}{j+q}\bigr)\xi\biggr)
\end{equation*}
The maximum of the term in  brackets is attained at
$j(1-\e_1)^{-1}$. For $\e_1 > \e$, so that
$j(1-\e_1)^{-1}<\eta$, we conclude that on the
interval $[\eta, \infty)$ the integrand does not
exceed
\begin{equation*}
\biggl[\eta^j e^{-(1-\e_1)\eta}\biggr]
\exp\biggl(-\bigl(\e_1-\frac{2}{j+q}\bigr)\xi\biggr).
\end{equation*}
Integrating, we get
\begin{equation*}
\mu_j\le (j+q)^{2q} \frac{q!}{(j+q)!} \frac{1}{\e_1-
2(j+q)^{-1}} \eta^j e^{-(1-2(j+q)^{-1})\eta}.
\end{equation*}
As at the first step of the proof, by the Stirling
formula we obtain
\begin{equation*}
\mu_j\le C(j+q)^{q-1/2}q e^q  \frac{1}{\e_1-
2(j+q)^{-1}}   e^{2(j+q)^{-1}\eta} \tau_j(\eta),
\end{equation*}
with the function $\tau_j(\eta)$ defined in
\eqref{aj:eq}. To estimate it, we rewrite
\begin{equation*}
\tau_j(\eta) = \exp\biggl[ - \int_j^\eta
\biggl(1-\frac{j}{s}\biggr) ds \biggr].
\end{equation*}
Choose an $\e_1\in (0, \e)$ and estimate:
\begin{equation*}
\tau_j(\eta) \le  \exp\biggl[ -
\int_{j(1-\e_1)^{-1}}^\eta \biggl(1-\frac{j}{s}\biggr)
ds \biggr] \le \exp\biggl[ -
\e_1\int_{j(1-\e_1)^{-1}}^\eta ds \biggr] \le
\exp\biggl( - \e_1 \eta \frac{\e-\e_1}{1-\e_1}\biggr).
\end{equation*}
This shows that $\mu_j$ tends to zero very fast as $\eta\to\infty$ and
$j\le (1-\e)\eta$, that is
$\l_j \to 1$ as $t\to \infty$, uniformly for $j\le (1-\e)\eta$.
 Summarizing  the above calculations, we see that for
sufficiently large $t$ the following inequalities hold:
\begin{equation*}
\l_j(d_1, d_2)\equiv \l_j(0, d_2) - \l_j(0, d_1)
\begin{cases}
< \l, \ j < (1-\e) \eta_1,\\
>\l, \ (1+\e)\eta_1 < j < (1-\e) \eta_2,\\
< \l, \ j > (1+\e)\eta_2.
\end{cases}
\end{equation*}
Consequently, $ (1-\e)\eta_2 - (1+\e) \eta_1 \le n(\l,
S^{(t)}(W)) \le (1+\e)\eta_2 - (1-\e) \eta_1$ for
sufficiently large $t$. Passing to the limit, we
obtain
\begin{equation*}
\gm(\l; W)\ge (1-\e)\frac{B d_2^2}{2} - (1+\e)
\frac{Bd_1^2}{2},\; \GM(\l; W)\le
(1+\e)\frac{Bd_2^2}{2} - (1-\e) \frac{Bd_1^2}{2}.
\end{equation*}
Since $\e\in (0, 1)$ is arbitrary, this entails
\eqref{gm2:eq}.
\end{proof}

Now, using the additivity and the calculations for the
model operator we can find $\gm(\l; w_{m, l}X_{m, l})$
and  $\GM(\l; w_{m, l}X_{m, l})$ which turn out to be
equal.

\begin{lem}\label{single:lem}
If $\l \ge |w_{m, l}|^2$, then $\GM(\l, w_{m, l}X_{m,
l}) = 0$. For any $\l\in (0, |w_{m, l}|^2)$
\begin{equation*}
\gm(\l; w_{m, l}X_{m, l}) = \GM(\l; w_{m, l}X_{m, l})
= \frac{B}{2\pi} |\Om_{m, l}|.
\end{equation*}
\end{lem}

\begin{proof}
Since $n(\l,S^{(t)}(w_{m, l}X_{m,
l}))=n(\l|w_{m, l}|^{-2},S^{(t)}(X_{m, l}) )$, it
suffices to consider the case $w_{m,l}=1.$ The norm of
the operator $S^{(t)}(X_{m, l}) $ is not greater than
1,  and this takes care of the case $\l\ge1.$ Next
consider $X_m = \sum_{l=1}^N X_{m, l}$. It is clear
 that the asymptotic coefficients are the same for all
 sectors
 $X_{m,l}$, $l=1,\dots,N$. Therefore, by
 Lemma~\ref{additivity.lem},
 \begin{equation}\label{sector}
    \gm(\l+0;X_m)\le N\gm(\l;X_{m,l})\le
    N\GM(\l;X_{m,l})\le\GM(\l-0;X_m).
\end{equation}
Now \eqref{gm2:eq}, produces the required formula.
\end{proof}

Before we proceed to treating more general functions $W$,
we introduce, similarly to \eqref{lm:eq},
the appropriate asymptotic coefficients.
For  $V\in\plainL1(\R^2)$ and
$\l> 0$ we define $sup$- and  $sub$-
%and $zero$- level
measures of $V$:
\begin{equation}\label{apm:eq}
A^{(\pm)}(\l; V) = A(\l, \infty; \pm V) = |\{\bx: \pm V(\bx)> \l\}|,\
%A^{(0)}(\l; V) = |\{\bx: V(\bx)=\l\}|.
\end{equation}
With this notation, the result of Lemma \ref{single:lem} reads
\begin{equation}\label{single:eq}
\gm(\l; w_{m, l}X_{m, l}) = \GM(\l; w_{m, l}X_{m, l})
= \frac{B}{2\pi} A^{(+)}(\l; |w_{m, l}|^2 X_{m, l}),
\end{equation}
for all $\l >0$.

\begin{cor}\label{single1:cor}
Let $v_{m, l}$ be real for some $m, l$. Then
\begin{equation*}
\gn^{(\pm)}(\l; v_{m, l}X_{m, l}) = \GN^{(\pm)}(\l; v_{m, l}X_{m, l})
= \frac{B}{2\pi} A^{(\pm)}(\l;  v_{m, l} X_{m, l}),
\end{equation*}
for all $\l >0$.
\end{cor}

\begin{proof} Denote $V = v_{m, l} X_{m, l}$.
If $\pm v_{m, l} >0$, then obviously $\GN^{(\mp)}(\l; V) = 0$, and
$\GN^{(\pm)}(\l; V) = \GN^{(+)} (\l; \pm V) = \GM(\l; \sqrt{\pm V})$.
It remains to use formula \eqref{single:eq}.
\end{proof}

In order to proceed we need to establish the "continuity" of the coefficients
$A^{(\pm)}(\l; V)$ in the function $V$:

\begin{lem}\label{measure:lem}
Let $W_\d = V+V'_\d,\ \d\not=0$, where $V$,
$V'_\d\in \plainL1(\R^2)$, $V$ does not depend on
$\d$ and $\|V'_\d\|_1\to 0$ as $\d\to 0$. Then
\begin{equation}\label{elementary4:eq}\begin{split}
A^{(\pm)}(\l; V)\le &\ \liminf_{\d\to 0} A^{(\pm)}(\l; W_\d)
\\[0.2cm]
\le & \limsup_{\d\to 0} A^{(\pm)}(\l; W_\d) \le A^{(\pm)}(\l-0; V).
\end{split}
\end{equation}
\end{lem}

\begin{proof}
It suffices to
consider the sign "$+$". By definition, for any $\e\in
(0, \l)$,
\begin{equation*}\label{ele1:eq}\begin{split}
A^{(+)}(\l; W_\d) \le  &\ |\{\bx: W_\d(\bx)>
\l\}\cap \{\bx: |V'_\d(\bx)|\le \e\}|
+ |\{\bx: |V'_\d(\bx)|>\e\}| \\[0.2cm]
\le &\ A^{(+)}(\l - \e; V ) + \e^{-1}\|V'_\d\|_1.\end{split}
\end{equation*}
For the last estimate we have used the Chebyshev
inequality.
Passing to the limit as $\d\to 0$ and $\e\downarrow 0$,
we get the upper bound in
\eqref{elementary4:eq}. Similarly,
write
\begin{equation*}
A^{(+)}(\l; W_\d) \ge A^{(+)}(\l + \e; V)
-  \e^{-1}\|V'_\d\|_1.
\end{equation*}
Passing again to the limit as $\d\to 0$ and $\e\downarrow 0$,
we get the lower bound in
\eqref{elementary4:eq}.
\end{proof}

\subsection{Eigenvalue asymptotics for Toeplitz operators}
Corollary \ref{single1:cor} enables us to establish the spectral asymptotics for the
Toeplitz operator $T^{(t)}(V)$ with a piece-wise constant function $V$.

\begin{lem}\label{ToepNonsd:lem}
Let $V = \sum
v_{ml}X_{ml}$, where the sum is finite and $v_{ml}$
are real-valued.
% \begin{equation}\label{generality}
% \l>0 \textup{\textrm{ does not equal any of }} |v_{ml}|.\end{equation}
Then for any $q\ge0$
\begin{equation}\label{ToepAsNSD}
\frac{B}{2\pi} A^{(\pm)}(\l; V)
\le \gn^{(\pm)}(\l; V) \le \GN^{(\pm)}(\l; V) \le \frac{B}{2\pi} A^{(\pm)}(\l-0; V),
\end{equation}
for all $\l >0$.
\end{lem}

\begin{proof}   Use Lemma \ref{additivity1:lem} and Corollary \ref{single1:cor}.
\end{proof}

We are now in position to treat the general case.

\begin{thm}\label{central:thm}
For any $V\in \plainL1(\R^2)$
\begin{equation}\label{var:eq}
\frac{B}{2\pi}A^{(\pm)}(\l; V) \le  \gn^{(\pm)}(\l; V)
\le  \GN^{(\pm)}(\l; V) \le
\frac{B}{2\pi}A^{(\pm)}(\l-0; V).
\end{equation}
Moreover, if $\l$ is a generic value for $\pm V$, we
have the asymptotics
\begin{equation}\label{vargen:eq}
 \gn^{(\pm)}(\l; V)
=  \GN^{(\pm)}(\l; V) = \frac{B}{2\pi}A^{(\pm)}(\l;
V).
\end{equation}
\end{thm}

\begin{proof}
Since $n_-(\l; T^{(t)}(V)) = n_+(\l; T^{(t)}(-V))$, it
suffices to consider the sign $"+"$ only.

 For a positive $\d$ we find a sufficiently fine tiling
of the plane by annular sectors $\Omega_{m,l}$
and a piecewise constant function $\tilde V_\d$ represented in the form
\eqref{stepweight1} with a finite sum,
such that $\|V - \tilde V_\d\|_1< \d$.
Hence by Lemma \ref{base:lem}, $ \GN(\e; V-\tilde V_\d)\le B \d
(2\pi\e)^{-1}$. Furthermore,  by Lemma \ref{ToepNonsd:lem},
\begin{equation}\label{central:Vd}
\GN^{(+)}(\l; \tilde V_\d)\le  \frac{B}{2\pi}A^{(+)}(\l-0;\tilde V_\d),
\end{equation}
Thus by Lemma \ref{AsLemma}
\begin{equation*}
\GN^{(+)}(\l; V)\le \frac{B}{2\pi}\limsup_{\e\downarrow 0} \limsup_{\d\to 0}
A^{(+)} (\l-\e; \tilde V_\d).
\end{equation*}
By virtue of Lemma \ref{measure:lem}, the right hand side does not exceed
$B (2\pi)^{-1} A^{(+)}(\l-0; V)$, as required.

The corresponding lower bound
for $\gn^{(+)}(\l; V)$ is established similarly.
\end{proof}

\section{Reduction to Toeplitz operators}\label{reduction:sect}

Let $(\l_1, \l_2)\Subset (\L_\nu, \L_{\nu+1})$ with
some $\nu = -1, 0, \dots$, as in Theorem
\ref{main:thm}.  We denote
\begin{equation*}
a = {(\l_1+\l_2)}/{2},\ \ b = {(\l_2-\l_1)}/{2},\ \
H_{0a}=H_0-a, \ \ H_a=H_a(V^{(t)})=H+V^{(t)}-a.
\end{equation*}
 Our analysis of the counting
function $N(\l_1, \l_2; H)$ is based upon the obvious
relations (cf. \eqref{1:transf}):
\begin{equation}\label{square:eq}
N(\l_1, \l_2; H) = N(\l_1 - a, \l_2 - a; H - a) =
N(-b, b; H - a) = N\bigl(b^2; (H - a)^2\bigr),
\end{equation}
where, recall, $N(\eta;L)$ stands for the number of
eigenvalues of a  semi-bounded operator $L$ below
$\eta$. For methodological purposes we need to consider
an operator having somewhat more general form. Namely,
with $a$ fixed, for real-valued functions
$V\in\plainL2( \R^2), Z\in\plainL1(\R^2)$ we set
\begin{equation}\label{K:eq}
 L^{(t)}(V, Z) = H_{0a}^2 + V^{(t)} H_{0a} +
H_{0a}V^{(t)} + Z^{(t)},
\end{equation}
so that $(H-a)^2 = L^{(t)}(V, V^2)$.  Under
these conditions the operator is well
defined via the quadratic form
%$\bigl((H_{0a} + V^{(t)})u, (H_{0a}+V^{(t)})u \bigr)+(Z^{(t)}-{V^2}^{(t)}u,u)$
$$
\| H_{0a} u\|^2 +
2\re (H_{0a}u, V^{(t)} u) + (Z^{(t)}u,u)
$$
for $u\in \dom(H_0)$.
Using the diamagnetic inequality, it is easy to check that each term
of the perturbation
$V^{(t)} H_{0a} + H_{0a}V^{(t)} + Z^{(t)}$ is $H_{0a}^2$-form-compact.
The lowest point of
the essential spectrum of $L^{(t)}(V, Z)$ is
\begin{equation}\label{eta0:eq}
\eta_0 = \min_{q\ge 0} (\L_q - a)^2 >0.
\end{equation}
We are going to study the discrete spectrum of the
operator $L^{(t)}$, independently of its connection
with  $H^{(t)}$. The following important lemma
establishes asymptotic relations of the spectrum of
$L^{(t)}(V, Z)$ below $\eta_0$ and the spectra of
certain Toeplitz type operators.

\begin{lem}\label{BSred:lem}
Suppose that $V,Z\in\plainC\infty_0(\R^2)$, and let $\eta < \eta_0$.
Then for any $\e >0$ there exists an
$\tilde\e \downarrow 0$ as $ \e\downarrow 0$, and an integer $J_0=J_0(\e)$ such that
for all $J\ge J_0$ and the projection
\begin{equation}\label{cumul:eq}
P^{(J)} = \sum_{q\le J} P_q  \end{equation}
 we have
\begin{equation}\label{BSred:eq}
\begin{split}
& \gb\bigl(\eta; L(V,Z)\bigr) \ge \gb\bigl(\eta  \!-\!\tilde\e;
P^{(J)}L ((1\!-\!\e)V,(1\!-\!\e)Z) P^{(J)}\bigr),
\\[0.2cm]
 & \GB\bigl( \eta;  L(V, Z)\bigr)  \le
\GB\bigl(\eta\! +\!\tilde\e;
P^{(J)}L ((1\!+\!\e)V,(\!1+\!\e)Z)P^{(J)}\bigr).
\end{split}
\end{equation}
\end{lem}

\begin{proof}
We denote $P = P^{(J)}$, $Q = I-P$.
 Let $v$ be a constant such that
\begin{equation*}
\sup_{\bx}\bigl(|V(\bx)| + |\nabla V(\bx)| + |\Delta V(\bx)|+|Z(\bx)|\bigr) \le v.
\end{equation*}
 Our first aim is to bound the counting function
$N(\eta, L^{(t)})$ from above and from below by
similar counting functions for $PL^{(t)}P$ and
$QL^{(t)}Q$.

Recalling that $H_0 = \Pi_1^2+ \Pi_2^2$ with
$\Pi_k = -i\p_k -a_k,\ k=1, 2$, and denoting $\tilde V_k(\bx) = -i(\p_k V)(\bx)$,
we conclude that
\begin{align*}
[H_0, V^{(t)}] = &\
\sum_{k=1}^2 \bigl(\Pi_k [\Pi_k, V^{(t)}] + [\Pi_k, V^{(t)}]\Pi_k \bigr)\\[0.2cm]
= &\ t^{-1} \sum_{k=1}^2 \bigl(\Pi_k \tilde V_k^{(t)}
+ \tilde V_k^{(t)}\Pi_k \bigr) = t^{-1} \sum_{k=1}^2
\Pi_k \tilde V^{(t)}_k +  t^{-2}(\Delta V)^{(t)}.
\end{align*}
Consequently,
\begin{align*}
|(P[H_0, V^{(t)}]Qu, u)|
\le &\ t^{-1} \sum_{k=1}^2 |(\tilde V_k^{(t)}Qu, \Pi_k Pu)  + v t^{-2}\|u\|^2
\\[0.2cm]
\le &\ (2t^2)^{-1}(H_0 P u, P u) + 2^{-1} v^2\|Qu\|^2 + v t^{-2}\|u\|^2\\[0.2cm]
\le &\ (4t^2)^{-1}(H_{0a}^2 P u, P u) + 2^{-1} v^2\|Qu\|^2 + (v+a+1)t^{-2}\|u\|^2.
\end{align*}
Now it follows from \eqref{K:eq} that for any $\d >0$
\begin{align*}
\!\!\!|(PL^{(t)}Qu, u)| \le &\ 2|(V^{(t)}Qu,
H_{0a}Pu)|
+ |( P[H_0, V^{(t)}]Qu, u)|+ |(  Z^{(t)} Qu,  Pu)|\\[0.2cm]
\le &\  \d \|H_{0a}Pu\|^2 + \frac{v}{ \d} \|Qu\|^2 + \
(4t^2)^{-1}(H_{0a}^2 P u, P u)\\[0.2cm]
&  \ + 2^{-1} v^2\|Qu\|^2 + (v+a+1)t^{-2}\|u\|^2
 + \frac{\d}{2} \|  Pu\|^2 + \frac{v^2}{2\d}\|  Qu\|^2\\[0.2cm]
\le &\
\bigl(\d + (4t^2)^{-1}\bigr)(H_{0a}^2 P u, P u) + \frac{\d}{2}\|Pu\|^2\\[0.2cm]
&\ \ \ \ \ \ \ \ \ \ \  + 2^{-1} v^2\bigl( 1 + \d^{-1}\bigr)\|Qu\|^2
+ (v+a+1)t^{-2}\|u\|^2.
\end{align*}
Therefore,
\begin{align*}
(L^{(t)}u, u)\le &\ (L^{(t)}Pu, Pu) +
2\bigl(\d + (4t^2)^{-1}\bigr)(H_{0a} P u, P u) +   \d \|Pu\|^2\\[0.2cm]
&\ \ \ \ \ \ \ \ \ \ \  +  v^2\bigl( 1 +
\d^{-1}\bigr)\|Qu\|^2 + 2(v+a+1)t^{-2}\|u\|^2 +
(QL^{(t)}Qu, u),
\end{align*}
and
\begin{align*}
(L^{(t)}u, u)\ge &\ (L^{(t)}Pu, Pu) -
2\bigl(\d + (4t^2)^{-1}\bigr)(H_{0a}^2 P u, P u) -  \d \|Pu\|^2\\[0.2cm]
&\ \ \ \ \ \ \ \ \ \ \  - v^2\bigl( 1 +
\d^{-1}\bigr)\|Qu\|^2 - 2(v+a+1)t^{-2}\|u\|^2 +
(QL^{(t)}Qu, u).
\end{align*}
Denote
%\begin{gather*}
\begin{equation*}
\!\!\!\e_1 = 2\bigl(\d + (4t^2)^{-1}\bigr),\ \ \e_2 =
\d + 2(v + a +1)t^{-2},
%\\[0.2cm]
M = v^2\bigl( 1 + \d^{-1}\bigr) + 2(v+a+1)t^{-2},
%\end{gather*}
\end{equation*}
so that
\begin{equation}\label{3:splitting}
%\begin{cases}
\begin{split}
(L^{(t)}u, u)\le  (L^{(t)}Pu, Pu) + \e_1 (H_{0a}^2 P
u, P u) + \e_2\|Pu\|^2
 + \! M\|Qu\|^2 + \!(QL^{(t)}Qu,\! u), \\
 (L^{(t)}u, u)\ge  (L^{(t)}Pu, Pu) -
\e_1 (H_{0a}^2 P u, P u) -  \e_2\|Pu\|^2
 -  \!M\|Qu\|^2 + \!(QL^{(t)}Qu, \!u).
 %\end{cases}
 \end{split}
\end{equation}
Thus we have bounded  $L^{(t)}$ from above and from
below by orthogonal sums of operators acting in ranges
of $P$ and $Q$. We show now that for sufficiently
large $J$ the operators containing $Q$ do not
contribute to the $t$-asymptotics of the counting
function. We have
\begin{align*}
(QL^{(t)}Qu, u) = &\ \|H_{0a}Q u\|^2
+ (H_{0a}Qu, V^{(t)} Qu) + ( V^{(t)}Qu, H_{0a} Qu)+ ( Z^{(t)} Qu, Qu)\\[0.2cm]
\ge &\ \|H_{0a}Q u\|^2
- \frac{1}{2}\|H_{0a}Qu\|^2 - 2\|V^{(t)}Qu\|^2 - v \| Q u\|^2 \\[0.2cm]
 \ge &\ \frac{1}{2}\|H_{0a}Q u\|^2
 -   ( 2v^2+v )\|Qu\|^2\ge \frac12\|(\L_J-a)Qu\|^2-( 2v^2+v )\|Qu\|^2.
\end{align*}
So, for sufficiently large $J$  (depending on $M$ and
$\delta$), we have $QL^{(t)}Q - M Q \ge \eta Q$ and
therefore the operator $QL^{(t)}Q - M Q$ acting in the
range of $Q$ has no spectrum below $\eta$.

Finally, to estimate  the operator containing $P$ in \eqref{3:splitting}, we can write
\begin{equation*}
L^{(t)}(V, Z) \pm  \e_1 H_{0a}^2 = (1\pm \e_1)
L^{(t)}\bigl(V (1\pm\e_1)^{-1}, Z(1\pm
\e_1)^{-1}\bigr).
\end{equation*}
Therefore, for sufficiently large $J$ it follows from \eqref{3:splitting}
that
\begin{align*}
N(\eta; L^{(t)})\le &\ N\bigl(\eta+\e_2;
(1-\e_1) PL^{(t)}\bigl(V (1 -\e_1)^{-1}, Z(1 -\e_1)^{-1}\bigr) P \bigr),\\[0.2cm]
N(\eta; L^{(t)})\ge &\ N\bigl(\eta - \e_2; (1+\e_1)
PL^{(t)}\bigl(V (1 +\e_1)^{-1}, Z(1
+\e_1)^{-1}\bigr)P\bigr),
\end{align*}
and \eqref{BSred:eq} follows.
\end{proof}

Next we are going to derive an asymptotic estimate for
the operator $PL^{(t)}P$ as $t\to\infty$. We define
the  asymptotic coefficient  as
\begin{equation}\label{3:maintermY}
\COB(\eta; V, Z)
=\sum_{q=0}^\infty\frac{B}{2\pi}A^{(+)}((\L_q-a)^2-\eta,-2(\L_q-a)V-Z).
\end{equation}

Under the assumptions
$V\in\plainL2(\R^2)\cap\plainL1(\R^2)$,\
$Z\in\plainL1(\R^2)$ the above series is absolutely
convergent uniformly in $\eta$ varying on a compact
set. Indeed,
\begin{align*}
A^{(+)}\bigl((\L_q-a)^2-\eta, &\ -2(\L_q-a)V-Z\bigr)
\le A^{(+)}\bigl((\L_q-a)^2-\eta, |\L_q-a|\ |V| + |Z|\bigr)\\[0.2cm]
\le &\ A^{(+)}\bigl((\L_q-a)^2-\eta,
\frac{1}{2}(\L_q-a)^2  + 2|V|^2 + |Z|\bigr)\\[0.2cm]
\le &\ A^{(+)}\bigl(\frac{1}{2}(\L_q-a)^2-\eta, 2|V|^2 + |Z|\bigr).
\end{align*}
For sufficiently large $q$, by the Chebyshev
inequality the last term is bounded from above by
\begin{equation*}
\bigl(\frac{1}{2}(\L_q-a)^2-\eta\bigr)^{-1}\bigl(2\|V\|_2^2
+ \|Z\|_1\bigr).
\end{equation*}
This shows that indeed the series in
\eqref{3:maintermY} converges.
 Each term in \eqref{3:maintermY} is a
left semi-continuous non-decreasing function of $\eta$, and as a
result, $\COB(\eta; V, Z)$ is left semi-continuous  as
well.

Let $P^{(J)}$ be as
defined in \eqref{cumul:eq}.

\begin{lem}\label{pkp:lem}
Let $V\in \plainL1(\R^2)\cap \plainL2(\R^2)$, $Z\in
\plainL1(\R^2)$. Then for any $\eta < \eta_0$
\begin{equation}\label{pkp:lemEQ}
\begin{split}
&\liminf_{J\to\infty} \gb(\eta; P^{(J)} L(V,
Z)P^{(J)}) \ge \COB(\eta; V, Z),
 \\[0.2cm]
&\limsup_{J\to\infty}\GB(\eta; P^{(J)} L(V, Z)
P^{(J)}) \le \COB(\eta+0;V,Z ).
\end{split}
\end{equation}
In particular, if $(\L_q-a)^2-\eta$
 are generic values for $-2(\L_q-a)V-Z$ for all $q=0,1,\dots$, then
\begin{equation}\label{pkp:lemEQgen}
    \lim_{J\to\infty}\lim_{t\to\infty}t^{-2}N(\eta;P^{(J)}L^{(t)}(V, Z)
    P^{(J)})=\COB(\eta;V,Z ).
\end{equation}
\end{lem}

\begin{proof}  Fix $J$ and rewrite $PL^{(t)}P, P = P^{(J)}$ as
\begin{align*}
PL^{(t)}P = &\ \sum_{q\le J} \bigl((\L_q - a)^2 P_q +
2(\L_q-a)T_q^{(t)}(V)
+ T_q^{(t)}(Z)\bigr)\\[0.2cm]
+ &\ \sum_{q\le J} \sum_{q'\le J, q'\not=q}
\bigl((\L_{q'}-a)T_{q, q'}^{(t)}(V)  + T_{q, q'}^{(t)}(Z) \bigr) =:
G^{(t)} + \tilde G^{(t)}.
\end{align*}
By Corollary \ref{disjoint2:cor} and \eqref{gmadd:eq:K}
we have $\GR(\e; \tilde G) = 0$ for any $\e>0$.

The operator $G^{(t)}$ is a finite orthogonal sum of the operators
$G^{(t)}_q=(\L_q-a)^2+T_q^{(t)}(2(\L_q-a)V+Z)$
acting in Landau subspaces $\CL_q$, so it suffices to
consider each of them individually.
Recall that $\eta <\eta_0 \le (\L_q-a)^2$, so that $\eta  -
(\L_q-a)^2 <0$ and
\begin{equation*}
N\bigl(\eta;  G^{(t)}_q\bigr)
= n_+\bigl((\L_q-a)^2 - \eta; T_q^{(t)}(-2(\L_q-a)V - Z)\bigr).
\end{equation*}
Applying Theorem \ref{central:thm}, we obtain
\begin{equation*}%\label{3:K_q}
\GN^{(+)}((\L_q-a)^2 - \eta; -2(\L_q-a)V-Z)\le
\frac{B}{2\pi}A^{(+)}((\L_q-a)^2-\eta-0,
-2(\L_q-a)V-Z),
\end{equation*}
and therefore, after summation over $q$
\begin{equation*}
\GB(\eta; G)\le \COB(\eta+0; V, Z),
\end{equation*}
see definition \eqref{3:maintermY}. Applying Lemma
\ref{AsLemma1} to the sum $G + \tilde G$, we obtain
the sought upper bound. The  lower bound in Lemma is
derived in the same way.
\end{proof}

\begin{cor}\label{redsmooth:cor}
Suppose that $V, Z$ are as in Lemma \ref{BSred:lem} and that $\eta < \eta_0$.
Then
\begin{equation}\label{redsmooth:eq}
 \COB(\eta; V, Z ) \le \gb(\eta; L( V, Z))
  \le
 \GB(\eta;   L(V, Z))  \le \COB(\eta + 0; V, Z).
\end{equation}
\end{cor}

\begin{proof}
According to \eqref{BSred:eq} and \eqref{pkp:lemEQ},
\begin{equation*}
 \gb(\eta;  L(V, Z) )
 \ge \COB\bigl(\eta - \tilde\e; V(1-\e), Z (1-\e)\bigr),
\end{equation*}
where $\tilde\e\downarrow 0$ as $\e\downarrow 0$. Passing to the
limit as $\e\downarrow 0$ with the help of \eqref{elementary4:eq},
we arrive at the required lower bound. The  upper
bound is obtained in the same way.
\end{proof}

\section{Reduction to a smooth
potential}\label{nonsmooth}

\subsection{Further estimates for the operator $L^{(t)}(V, Z)$}
In this section we continue the study of the operator $L^{(t)}(V, Z)$.
Our aim now is to extend Corollary
\ref{redsmooth:cor} to non-smooth functions $V$ and $Z$.

Recall again the notation
\begin{equation*}
L^{(t)}=L^{(t)}(V, Z) = H_{0a}^2 + V^{(t)} H_{0a} + H_{0a} V^{(t)} + Z^{(t)}
\end{equation*}
for some real-valued functions
$V\in \plainL1(\R^2)\cap\plainL2(\R^2)$, $Z\in\plainL1(\R^2)$.
We start with an eigenvalue estimate for the operator $L^{(t)}$.

\begin{lem}\label{nonsmooth:lem}
Suppose that $V\in\plainL2(\R^2)$ and $Z\in \plainL1(\R^2)$.
Then
\begin{equation}\label{cont:eq}
\GB(\eta; L(V, Z))\le \frac{C}{(\eta_0-\eta)^2} (\|V\|_2^2 + \|Z\|_1),
\end{equation}
for all $\eta < \eta_0$ and some constant $C$ independent of $\eta$.
\end{lem}

\begin{proof}
Since $|(V^{(t)}u,H_{0a}u)|\le \frac{1}{2z}\|V^{(t)}u\|^2
+\frac{z}{2}\|H_{0a}u\|^2$, we have for
$z\in (0, 1)$:
\begin{equation*}
L^{(t)}(V, Z)\ge (1-z) H_{0a}^2 - z^{-1}(V^{(t)})^2 - |Z^{(t)}|.
\end{equation*}
We chose $z = \frac{\eta_0-\eta}{2\eta_0}$ so that
$(1-z)H_{0a}^2 - \eta>0$ and $(1-z)H_{0a}^2 -
\eta\ge c(\eta_0-\eta)H_{0, -1}^2$, and denote
\begin{equation*}
Y_{a, \eta}^{(t)} = \bigl((1-z)H_{0a}^2 - \eta\bigr)^{-\frac{1}{2}}
\bigl(z^{-1}(V^{(t)})^2 +
|Z^{(t)}|\bigr)^{\frac{1}{2}}.
\end{equation*}
 By the Birman-Schwinger principle we have
\begin{equation*}
N\bigl(\eta; L^{(t)} (V, Z)\bigr) \le
n\bigl(1; (Y_{a,\eta}^{(t)})^* Y_{a, \eta}^{(t)}\bigr).
\end{equation*}
By the diamagnetic inequality we have
\begin{align*}
\|Y_{a, \eta}^{(t)}\|^2_{\GS_2}\le &\ C(\eta_0-\eta)^{-1}
\| Y_{-1, 0}^{(t)}\|^2_{\GS_2}\\[0.2cm]
\le &\ C (\eta_0-\eta)^{-1}\| (-\Delta+I)^{-1}\bigl(z^{-1}(V^{(t)})^2
+ |Z^{(t)}|\bigr)^{\frac{1}{2}} \ \|^2_{\GS_2}\\[0.2cm]
\le &\ C'(\eta_0-\eta)^{-2} t^2\bigl( \|V\|_2^2+ \|Z\|_1\bigr).
\end{align*}
Therefore $(Y_{a, \eta}^{(t)})^*Y_{a, \eta}^{(t)}$ is a trace class operator and
\begin{equation*}
\|(Y_{a, \eta}^{(t)})^*Y_{a, \eta}^{(t)}\|_{\GS_1}
\le C (\eta_0-\eta)^{-2} t^2\bigl( \|V\|_2^2+ \|Z\|_1\bigr).
\end{equation*}
The inequality \eqref{cont:eq} follows now by applying
the bound $n (\l; K)\le \l^{-1} \|K\|_{\GS_1}$.
\end{proof}

\begin{thm}\label{KVZ:thm}
 Let $V\in\plainL1(\R^2)\cap\plainL2(\R^2)$
and $Z\in\plainL1(\R^2)$.
Then formula \eqref{redsmooth:eq} holds for any $\eta<\eta_0$.
\end{thm}

\begin{proof}
The idea is to approximate $V, Z$ by smooth functions and then use
Lemma \eqref{AsLemma2} to show that the error does not contribute.

For a fixed $\d>0$ represent $V, Z$ as $V = V'_\d + V''_\d$, $Z = Z'_\d+Z''_\d$
so that $V'_\d, Z'_\d\in \plainC\infty_0(\R^2)$ and
$$
 \|V''_\d\|_1+\|V''_{\d}\|_2^2 + \|Z''_{\d}\|_1<\d.
$$
We apply Lemma~\ref{AsLemma2}. Clearly,
\begin{gather*}
L (V, Z) = L(0, 0) + Y'_\d + Y''_\d,\\[0.2cm]
Y'_\d = V'_\d H_{0a} + H_{0a} V'_\d + Z'_\d,\
Y''_\d = V''_\d H_{0a} + H_{0a} V''_\d + Z''_\d.
\end{gather*}
$L (V, Z) = L(0, 0) $. According
to \eqref{cont:eq}
\begin{gather*}
\GB(\tau; L(0, 0) + M Y''_\d)\le C\frac{C(1+M^2) \d}{(\eta_0-\tau)^2}
\end{gather*}
for any $\tau < \eta_0$ and any $M\in R$,
and thus the condition \eqref{AsLemma2:eq} is fulfilled.
Now it follows from Corollary \ref{redsmooth:cor} and
Lemma \ref{AsLemma2} that
\begin{equation*}
\GB(\eta; L(V, Z))\le \lim_{\e\downarrow 0}
\limsup_{M\downarrow 1} \limsup_{\d\to 0}
\COB(\eta+\e; M V'_\d, MZ'_\d)
\end{equation*}
By Lemma \ref{measure:lem} the right hand side does
not exceed $\COB(\eta+0;  V, Z)$, as claimed.
Similarly, one obtains the appropriate lower bound for
$\gb(\eta, L(V, Z))$.
\end{proof}

\subsection{The general asymptotic estimate: proof of Theorem \ref{main:thm}}
Recall that the coefficients
$\COA$ and $\COB$ are defined in \eqref{1:fullcoeff} and \eqref{3:maintermY}
respectively.
By \eqref{square:eq} we have
$N(\l_1, \l_2; H^{(t)}) = N(b^2; L^{(t)}(V, V^2))$. Observe
that $\COA(\l_1, \l_2; V) = \COB(b^2; V, V^2)$. It remains to apply Theorem
\ref{KVZ:thm}.

\bibliographystyle{amsplain}

\providecommand{\bysame} {\leavevmode\hbox
to3em{\hrulefill}\thinspace}

\end{document}